\documentclass{amsart}
\usepackage{amsmath, amssymb,epic,graphicx,mathrsfs,enumerate}
\usepackage[all]{xy}

\usepackage{amsthm}
\usepackage{amssymb}
\usepackage{latexsym}
\usepackage{longtable}
\usepackage{epsfig}
\usepackage{amsmath}
\usepackage{hhline}

 \DeclareMathOperator{\perm}{Sym}
 \DeclareMathOperator{\soc}{soc}

 \DeclareMathOperator{\frat}{Frat}

\DeclareMathOperator{\gl}{GL}

 \DeclareMathOperator{\h}{{H^2}}

\renewcommand{\emptyset}{\varnothing}

\newcommand{\pq}{$p\cdot q$ }

\newtheorem{thm}{Theorem}

 \newtheorem{lemma}[thm]{Lemma}
\newtheorem{prop}[thm]{Proposition} 
 \newtheorem{defn}[thm]{Definition}
\newtheorem{question}[]{Question} 
\numberwithin{equation}{section}

\renewcommand{\footnote}{\endnote}
\newcommand{\ignore}[1]{}\makeglossary

\begin{document}
	\bibliographystyle{amsplain}
	\title[The order of the non-Frattini elements]{On the orders of the non-Frattini elements\\ of a finite group}

\author{Andrea Lucchini}
\address{
	Andrea Lucchini\\ Universit\`a degli Studi di Padova\\  Dipartimento di Matematica \lq\lq Tullio Levi-Civita\rq\rq \\email: lucchini@math.unipd.it}

	\begin{abstract}Let $G$ be a finite group and let $p_1,\dots,p_n$ be distinct primes. If $G$ contains an element of order $p_1\cdots p_n,$ then there is an element in $G$ which is not contained in the Frattini subgroup of $G$ and whose order is divisible by  $p_1\cdots p_n.$
	\end{abstract}
	\maketitle
\section{Introduction}	
Let $G$ be a finite group $G.$ We say that $g\in G$ is a non-Frattini element of $G$ if $g\notin \frat(G)$, where $\frat(G)$ denotes the Frattini subgroup of $G.$ Let $d(G)$  be the smallest cardinality of a generating set of $G.$ In \cite {hidden} L. G. Kov\'acs,  J. Neub\"user and  B. H. Neumann proved that the set $G\setminus \frat(G)$ of the non-Frattini elements of $G$ coincides with the set $\kappa_{d(G)+1}(G)$ of those elements $g$ of $G$ that
are not omissible from some family of $d(G)+1$ generators of $G.$  In the same paper they noticed that if a prime $p$ divides $|G|,$  then $p$ divides the order of some element of $G\setminus \frat(G)=\kappa_{d(G)+1}(G).$ This follows immediately from the fact that if a prime $p$ divides $|G|,$ then $p$ divides also $|G/\frat(G)|$ \cite[III Satz 3.8]{hup}. In other words the set $\pi(G)$ of the prime divisors of $|G|$ can be deduced from the knowledge of the non-Frattini elements of $G$ or, equivalently, looking at the elements which appear in the minimal generating sets of cardinality $d(G)+1.$ One can ask whether some more elaborated  information about the arithmetical properties of the orders of the elements of $G$ can be recovered from   $G\setminus \frat(G).$ 
Looking for results in this direction, we consider in this note the prime graph of $G.$

\

If $G$ is a finite group,  its prime graph $\Gamma(G)$ is defined as follows: its vertices are
the primes dividing the order of $G$ and two vertices $p, q$ are joined by an edge if \pq divides the order of some  element in $G.$ It has been introduced by Gruenberg and Kegel in the 1970s and  studied extensively in recent years (see for examples \cite{kon}, \cite{wil}, \cite{zav}). We consider now the subgraph $\tilde \Gamma(G)$ of $\Gamma(G)$ (called the non-Frattini prime graph) defined by saying that two vertices $p$ and $q$ are joined by an edge if and only if \pq divides the order of some element of $G\setminus \frat(G).$ Our main result is the following.

\begin{thm}\label{pg}
Let $G$ be a finite group. Then the prime graph $\Gamma(G)$ and the non-Frattini prime graph $\tilde\Gamma(G)$ coincide.
\end{thm}

Notice that, although $\pi(G)=\pi(G/\frat(G))$, it is not in general true that $\Gamma(G)=\Gamma(G/\frat(G)).$ Consider for example $G=\langle a,b \mid a^3=1, b^4=1, a^b=a^{-1}\rangle.$ We have $|G|=12$ and $|ab^2|=6,$ so $\Gamma(G)$ is the complete graph on the two vertices 2 and 3. However $\frat(G)=\langle b^2\rangle$ and $G/\frat(G)\cong \perm(3),$ hence $\Gamma(G/\frat(G))$ consists of two isolated vertices. 

\

Clearly if \pq divides the order of some elements of $G$, then $G$ contains an element of order \pq. This is no more true for $G\setminus \frat(G).$
For example all the elements of order \pq in a  cyclic group $G$ of order $p^2\cdot q^2$ are contained in $\frat(G).$ A question that is not easy to be answered is, in the case when $G\setminus \frat G$ contains an element of order divisible by $p\cdot q,$   whether this element could be chosen so that $p$ and $q$ are the unique prime divisors of its order. We prove that this is true for soluble groups and we provide a reduction to this question to a problem concerning the finite nonabelian simple groups and their representations.

\section{Proofs and remarks}
Given an element $g$ in a finite group $G,$ we will denote by $\pi(g)$ the set of the prime divisors of $|g|.$ Theorem \ref{pg} is a consequence of the following stronger result.

\begin{thm}\label{main}
	Let $G$ be a finite group and suppose that $\pi=\{p_1,\dots p_n\}$ is a subset of the set $\pi(G)$ of the prime divisors of $|G|.$ Set $\pi^*=\pi$ if
	$G$ is soluble, $\pi^*=\pi\cup\{2\}$ otherwise. 
	If $G$ contains an element $g$ of order $p_1\cdots p_n,$ then there exists an element $\gamma$ in $G\setminus \frat (G)$ such that $\pi\subseteq \pi(\gamma) \subseteq \pi^*.$
\end{thm}

\begin{proof}
Choose   $g\in G$ with $|g|=p_1\cdots p_n.$ We may assume $g\in \frat(G),$ otherwise we have done. Let $N$ be a minimal normal subgroup of $G.$ Consider the element $\bar g=gN$ of the factor group $\bar G=G/N$ and let $M/N=\frat (G/N).$ If $|\bar g|=|g|,$ then by induction there exists $x\in G$ such that $\bar x= xN \not\in M/N$ and $\pi \subseteq \pi(\bar x) \subseteq \pi^*.$ We may choose $x$ such that $\pi(x)=\pi(\bar x).$
Since $x\notin M$ and, by \cite[5.2.13 (iii)]{rob}, $N\frat(G)\leq M$, we have that $x \notin \frat(G)$ and we are done. So we may assume that there exists $i\in\{1,\dots,n\}$ with $1\neq g^{p_i}\in N$. If $N\not\leq \frat (G)$, then $N\cap \frat(G)=1$, hence $g^{p_i}\notin \frat(G).$ This implies $g\notin \frat(G)$, a contradiction.

We remain with the case in which for every minimal normal subgroup $N$ of $G$, we have that $N$ is contained in $\frat(G)$ and $\langle g \rangle \cap N\neq 1.$  Since $\frat(G)$ is nilpotent, we deduce that $\pi(\frat(G))\subseteq \pi.$ On the other hand we must have that $\pi \subseteq \pi(\frat(G))$, otherwise $g\notin \frat(G).$ 
So $\pi(\frat(G))=\pi$ and $\frat(G)=P_1 \times \dots \times P_n$, where for each $i\in \{1,\dots,n\}$, $P_i$ is a $p_i$-group.  Let $i \in \{1,\dots,n\}$ and
 $N$ a minimal normal subgroup of $G$ with $N\leq P_i.$ Assume $N\neq P_i.$ In this case, choose $x\in P_i\setminus N,$ and take $y=xg^{p_i}$. The element $\bar y=yN$ of the factor group $G/N$ has order divisible by $p_1\cdots p_n$ and, as in the first paragraph of this proof, this allows us to conclude that $G\setminus \frat(G)$ contains an element of order divisible by $p_1\cdots p_n$. So we may assume $N=P_i$. In particular $P_i$ is an irreducible $G$-module, for every $i\in \{1,\dots,n\}.$ 
 
First assume that $G/\frat(G)$ contains an abelian minimal normal subgroup $M/\frat(G).$ There exists a prime $p$ such that $M/\frat(G)$ is a $p$-group. If $p\notin \pi,$ then $\frat(G)$ is a normal $\pi$-Hall subgroup of $M$ and  therefore, by the Schur-Zassenhaus Theorem, $\frat(G)$ has a complement, say $K$, in $M$ and all these complements are conjugate in $M$.  By the Frattini's Argument, $G=\frat(G)N_G(K),$ hence $G=N_G(K),$ so $K$ is a nontrivial normal subgroup of $G$. However all the minimal normal subgroups of $G$ are contained in $\frat(G)$ and $\frat(G)\cap K=1,$ a contradiction. So $p=p_i$ for some $i\in 
\{1,\dots,n\}.$ Let $Q_i=\prod_{j\neq i}P_j$ and let $T_i$ be a Sylow $p_i$-subgroup of $M.$ Again by the Frattini Argument, $G=MN_G(T_i)=Q_iT_iN_G(T_i)=Q_iP_iN_G(T_i)=\frat(G)N_G(T_i),$ so $N_G(T_i)=G$ and $M=Q_i \times T_i.$ Take $x\in T_i\setminus P_i$ and consider $\gamma=xg^{p_i}.$ Since $g^{p_i}\in Q_i$, we have $|\gamma|=|x||g^{p_i}|$ and consequently $\pi(\gamma)=\pi.$

Finally assume  that $M/\frat(G)$ is a nonabelian minimal normal subgroup of $G/\frat(G)$ and let $x=z\frat(G)$ be an element of $M/\frat(G)$ 
 of order 2.  We may assume $|z|=2^c$ for some $c>0.$
Let $I$ be the subset of  $i \in \{1,\dots,n\}$ consisting of the indices $i$ such that $p_i$ is odd. If $i\in I$ and $M\leq C_G(P_i)$, let $x_i$ be an arbitrarily chosen nontrivial element of $P_i.$ Assume that $i\in I$ and that $M\not\leq C_G(P_i).$ In this case $M/\frat(G)$ is isomorphic to a subgroup of $\gl(P_i)$. It can be easily seen that if $y\in \gl(P_i)$ has order 2, then either 1 is an eigenvalue of $y$ or $y$ is the scalar multiplication by -1. Since $Z(M/\frat(G))=1,$ we deduce that $z$ fixes a non-trivial element of $x_i$ of $P_i.$ Now let $x=\prod_{i\in I}x_i$ and take $\gamma=zx.$ Since $|\gamma|=2^c
\prod_{i\in I}p_i$, we conclude $\pi(\gamma)=\pi^*.$
\end{proof}

Given a pair $(p,q)$ of distinct prime divisors of the order of a finite group $G,$ let
$$\Omega_{p\cdot q}(G)=\{g\in G\mid p\cdot q \text{ divides } |g|\}, \quad \Omega^*_{p\cdot q}(G)=\Omega_{p\cdot q}(G)\setminus \frat(G).$$ 
Moreover, denote by $\Omega^{**}_{p,q}(G)$ the set of the elements $g\in \Omega^*_{p,q}(G)$ whose order is not divisible by any prime different from $p$ and $q.$
One can ask the following question.

\begin{question}\label{que} Is it true that if $\Omega_{p,q}(G)\neq \emptyset,$ then  $\Omega^{**}_{p,q}(G)\neq \emptyset$ ?
\end{question}

By Theorem \ref{main}, Question \ref{que} has an affirmative answer if $2\in \{p,q\}.$

\begin{defn}
Let $S$ be a finite nonabelian simple group, $P$ a faithful irreduble $S$-module of $p$-power order and $Q$ a faithful irreduble $S$-module of $q$-power order. We say that $(S,P,Q)$ is a $(p,q)$-Frattini triple if the following conditions are satisfied:
\begin{enumerate}
	\item $\Omega_{p\cdot q}(S)=\emptyset;$
	\item $\h(S,P)\neq 0;$ 
	\item $\h(S,Q)\neq 0;$
	\item $C_P(s)=0$ for every nontrivial element  $s$ of $S$ with $q$-power order;
	\item $C_Q(s)=0$ for every nontrivial element  $s$ of $S$ with $p$-power order.
\end{enumerate}
\end{defn}

Notice that if $(S,P,Q)$ is a $(p,q)$-Frattini triple then we may construct a Frattini extension $G$ of $P\times Q$ by $S.$  

\begin{lemma}
If $(S,P,Q)$ is a $(p,q)$-Frattini triple and $G$ is a Frattini extension $G$ of  $P\times Q$ by $S,$  then $\Omega^{**}_{p,q}(G)= \emptyset.$
\end{lemma}
\begin{proof}
Let $F=\frat(G)=P\times Q.$ Assume that $g$ is an element of $G$ with $|g|=p^a\cdot q^b,$ for some positive integers $a$ and $b.$ We can write $g$ as $g=xy$ where $x$ has order $p^a,$ $y$ has order $q^b$ and $x$ and $y$ commute.  By (1),   $S\cong G/F$ does not contain elements of order \pq, hence either $x\in P$ or $y\in Q.$ It is not restrictive to assume $x\in P.$ But then $y$ is a $q$-element element of $C_G(x)$, i.e. $C_P(yF)\neq 0.$ By (4), this implies $y\in Q$ and consequently $g=xy \in F.$
\end{proof}

\begin{prop}\label{minimal}
Let $p$ and $q$ be two odd primes. Assume that $G$ is a finite group of minimal order with respect to the property that  $\Omega_{p,q}(G)\neq \emptyset$ but 
$\Omega^{**}_{p,q}(G)=\emptyset$. Then there exists a $(p,q)$-Frattini triple $(S,P,Q)$ such that $G$ is a Frattini extension $G$ of by $P\times Q$ by $S.$ 
\end{prop}
\begin{proof}
Let $X=G/\frat(G).$ As in the proof of Theorem \ref{main}, $\frat(G)=P\times Q$, where $P$ and $Q$ are, respectively, the Sylow $p$-subgroup and the Sylow $q$-subgroup of $\frat(G).$ Moreover $P$ and $Q$ are the unique minimal normal subgroups of $G.$ Now let $C=C_G(P).$ Clearly $Q\leq C$. Let $x\in C$ such that $|x|$ is a nontrivial $q$-power and let $1\neq y\in P.$ The order of $xy$ is divisible by \pq so  $xy \in \frat(G)$  i.e. $x \in Q.$ Hence $Q$ is a normal $q$-Sylow subgroup of $C$ and  therefore, by the Schur-Zassenhaus Theorem, $Q$ has a complement, say $K$, in $C$ and all these complements are conjugate in $C$. By the Frattini's Argument, $G=QN_G(K)=\frat(G)N_G(K),$ hence $G=N_G(K)$ and $C=Q\times K.$ In particular $K\leq C_G(Q)$, and, repeating the same argument as before, we deduce
that $P$ is a normal Sylow subgroup of $K$. Again by the Schur-Zassenhaus Theorem, $P$ is complemented in $K$: so we have $C=(P\times Q)\rtimes T=\frat(G)\rtimes T$
for a suitable subgroup  $T$ whose order is coprime with $p\cdot q$. By the Frattini's Argument, $G=\frat(G)N_G(T)$ hence $G=N_G(T)$ so $T$ is normal in $G$.
However all the minimal normal subgroups of $G$ are contained in $\frat(G)=P\times Q$ and $\frat(G)\cap T=1.$ So it must be $T=1,$ i.e. $C_G(P)=\frat(G).$ With a similar argument we deduce that $C_G(Q)=\frat(G)$ so $P$ and $Q$ are faithful nontrivial irreducible $X$-module, setting $X=G/\frat(G).$ By \cite[5.2.13 (iii)]{rob}, 
$\frat(G/P) =\frat(G)/P\cong Q$, so $G/P$ is a non-split extension of $Q$ by $X$ and consequently $\h(X,Q)\neq 0$ and, similarly, $\h(X,P)\neq 0.$
 Let $Y$ be a non-trivial normal subgroup of $X.$ Since $C_P(Y)$ is $X$-invariant and $P$ is an irreducible and faithful $X$-module, 
it must be  $C_P(Y)=0$, hence, by \cite[Corollary 3.12 (2)]{GKKL}, if $p$ would not divide $|S|,$ then $\h(X,P)=0$, so $p$ (and similarly $q$) must divide $|Y|$.  Let $T$ be a Sylow $p$-subgroup of $G$. If there exists $t\in T\setminus P$ commuting with a non-trivial element $y$ in $Q,$ then $ty\in G\setminus \frat(G)$ and $|ty|=|t||y|=p^aq$ for some $a\in \mathbb N,$ against our assumption. But then $T/P\cong T\frat(G)/\frat(G)$ is a fixed-point-free group of automorphisms of $Q$, and therefore, by \cite[10.5.5]{rob}, $T/\frat(G)$ is a cyclic group. Similarly, a Sylow $q$-subgroup of $G/\frat(G)$ is cyclic. Now let $S=\soc X.$
Since $p$ divides the order of every minimal normal subgroup of $X$, $S=S_1\times \dots \times S_t,$ where, for each $1\leq i \leq t,$ $S_i$ is a simple group whose order is a multiple of $p.$ However a Sylow $p$-subgroup of $S$ is cyclic, hence we must have $t=1$, i.e. $S$ is a simple group. Moreover by \cite[Lemma 5.2]{GKKL}, $S$ is nonabelian. Let now $U$ be a Sylow $p$-subgroup of $X$. Since $U$ is cyclic, if $U\not\leq S,$ then $U\cap S\leq \frat(U).$ By a 
theorem of Tate (see for instance \cite [p. 431]{hup}), $S$ would be $p$-nilpotent, a contradiction. Hence $U\leq S$ and therefore, by \cite[Lemma 3.6]{GKKL}, the restriction map $\h(X,P)\to \h(S,P)$ is an injection. So in particular $\h(S,P)\neq 0$ and similarly $\h(S,Q)\neq 0.$ This implies that there exist an irreducible $S$-submodule $P^*$ of $P$ and an irreducible $S$-submodule $Q^*$ of $Q$ such that $H^2(S,P^*)\neq 0$ and $H^2(S,Q^*)\neq 0.$ We have that $(S,P^*,Q^*)$ is a $(p,q)$-Frattini triple.  
\end{proof}

Let $\mathcal S_{p,q}$ be the set of the nonabelian simple groups $S$ for which there exist $P$ and $Q$ such that $(S,P,Q)$ is a a $(p,q)$-Frattini triple. 
From Proposition \ref{minimal} and its proof, the following can be easily deduced.

\begin{thm}
Let $p$ and $q$ be two distinct odd primes. Assume that $G$ is a finite group with  $\Omega_{p\cdot q}(G)\neq\emptyset.$ If no composition factor of $G$ is in
 $\mathcal S_{p,q},$ then  $\Omega^{**}_{p,q}(G)\neq\emptyset.$
\end{thm}

It seems a difficult problem to determine whether 	$\mathcal S_{p,q}$ is non empty. The following remark can give some help in dealing with this question.
Following  \cite{dixon} and \cite{ig}, we say that
a subset $\{g_1, \ldots , g_d\}$ of a finite group $G$  invariably generates $G$ if
$\{g_1^{x_1}, \ldots , g_d^{x_d}\}$ generates $G$ for every choice of $x_i \in G$.
\begin{prop}
Assume that $S\in \mathcal S_{p,q}$. If no proper subgroup of $S$ is isomorphic to a group in $\mathcal S_{p,q}$, then there exist $x$ and $y$ in $S$ such that:
\end{prop}
\begin{enumerate}
\item $\langle x \rangle$ is a Sylow $p$-subgroup of $S$;
\item $\langle y \rangle$ is a Sylow $q$-subgroup of $S$;
\item $\{x,y\}$ invariably generates $S.$
\end{enumerate}

\begin{proof}
Let $(S,P,Q)$ be a $(p,q)$-Frattini triple and let $X$ and $Y$ be, respectively, a Sylow $p$-subgroup and a the Sylow $q$-subgroup of of $S.$  The subgroup $X$ is a fixed-point-free group of automorphisms of $Q$, and therefore, by \cite[10.5.5]{rob}, $X$ is a cyclic group. Similarly $Y$ is  a cyclic group.
Let $X=\langle x \rangle$ and $Y=\langle y \rangle.$ Assume, by contradiction, that $\{x,y\}$ does not invariably generate $S.$ Then there exist $s, t\in S$ such that $H=	\langle x^s, y^t\rangle$ is a proper subgroup of $S.$ Since $H$ contains both a Sylow $p$-subgroup and a Sylow $q$-subgroup of $G$,  by \cite[Lemma 3.6]{GKKL} the restriction maps $\h(S,P)\to \h(H,P)$ and $\h(S,Q)\to \h(H,Q)$ are injective. Hence $\h(H,P)\neq 0$ and $\h(H,Q)\neq 0.$ Arguing as in the proof of Proposition \ref{minimal}, we deduce that $T=	\soc H$ is a finite nonabelian simple group and $T\in \mathcal S_{p,q}$, against our assumption.
\end{proof}

\end{document}